\documentclass[11pt, a4paper, article]{amsart}
\usepackage{graphicx, xcolor, amsmath, amssymb}

\newtheorem{theorem}{Theorem}[section]
\newtheorem{proposition}[theorem]{Proposition}

\newtheorem{lemma}[theorem]{Lemma}

\newtheorem{example}[theorem]{Example}
\newtheorem{remark}[theorem]{Remark}

\title{\sf\large Small, medium and large shock \\
waves for radiative Euler equations}
\author{Corrado Mascia}
\address{Dipartimento di Matematica ``G. Castelnuovo''\\ 
Sapienza -- Universit\`a di Roma\\ 
P.le Aldo Moro, 2 - 00185 Roma (ITALY)}

\begin{document}

\baselineskip15pt
\thispagestyle{empty}

\begin{center}\sf
{\Large Small, medium and large shock waves}\vskip.2cm
{\Large for non-equilibrium radiation hydrodynamics}\vskip.25cm
{\tt \today}\vskip.2cm
Corrado MASCIA\footnote{Dipartimento di
  Matematica ``G. Castelnuovo'', Sapienza -- Universit\`a di Roma, P.le Aldo
  Moro, 2 - 00185 Roma (ITALY), \texttt{\tiny mascia@mat.uniroma1.it}
  {\sc and} {Istituto per le Applicazioni del Calcolo, Consiglio Nazionale delle Ricerche
  (associated in the framework of the program ``Intracellular Signalling'')}}
\end{center}
\vskip.5cm

\begin{quote}\footnotesize\baselineskip 12pt 
{\sf Abstract.} 
We examine the existence of shock profiles for a hyperbolic-elliptic system arising
in radiation hydrodynamics. 
The algebraic-differential system for the wave profile is reduced to a standard
two-dimensional form that is analyzed in details showing the existence of heteroclinic
connection between the two singular points of the system for any distance between
the corresponding asymptotic states of the original model.
Depending on the location of these asymptotic states, the profile can be either continuous
or possesses at most one point of discontinuity. 
Moreover, a sharp threshold relative to presence of an internal absolute maximum 
in the temperature profile --also called {\sf Zel'dovich spike}-- is rigourously derived.
\vskip.15cm

\vskip.15cm

\end{quote}

\section{Introduction}

Fluid dynamics equations support a class of special and significant solutions known as {\sf shock waves}.
These describe physical phenomena consisting in an abrupt transition from one state to another, that,
in the simplest setting, are described by a single jump connecting two different states and propagating,
in first approximation, with constant speed.
In formulas, these corresponds to special solutions of the system under consideration having the form
\begin{equation*}
	w(x,t)=W(x-ct)\qquad W(\pm\infty)=W_\pm,
\end{equation*}
where $w$ denotes the state variable, $W$ the shock profile, $c$ the velocity and $W_\pm$ the
asymptotic states.
The difference $[W]:=W_+-W_-$ and its length are generally used to distinguish between small,
$|W_+-W_-|\to 0$, and large shocks, $|W_+-W_-|\to \infty$.

In the relevant case of the Euler equations for gasdynamics in one space dimension, obtained considering
the conservation of mass, momentum and energy and neglecting higher order effect as viscosity, capillarity,
and thermal conductivity, traveling wave solutions propagate with a velocity dictated by the Rankine--Hugoniot
condition.
Moreover, they have piecewise constant profiles with a single jump point, where physical entropy increases
along the discontinuity.
The persistence of the presence of waves when considering also higher order effects has been  considered 
by several authors and in different context.
Among others, let us quote the classical reference \cite{Gilb51} where the existence of shock layers
for a model taking into account both viscosity and heat conductivity.

The present article fits into the strand of radiation hydrodynamics (see \cite{MihaMiha84, ZeldRaiz02}). 
Specifically, we consider the classic Euler equations coupled with an elliptic equation for an additional
variable describing the intensity of radiation; 
precisely, we deal with the hyperbolic-elliptic system
\begin{equation}\label{radeuler}
	\left\{\begin{aligned}
		&\partial_t \rho+\partial_x(\rho\,u)=0,\\
		&\partial_t (\rho\,u)+\partial_x(\rho\,u^2+p)=0,\\
		&\partial_t (\rho\,E)+\partial_x(\rho\,E\,u+p\,u-\sigma_s^{-1}\,\partial_x n)=0,\\
		&\partial_x^2 n = \tau^2(n - g(\theta))
	\end{aligned}\right.
\end{equation}
where the variables $\rho, u, p, E$ and $\theta$ represent density, velocity, pressure, specific total energy 
and temperature of the fluid under consideration.
The additional variable $n$ describes the specific intensity of radiation and the coupling is
governed by the constants $\tau, \sigma_s$ and by the function $g=g(\theta)$.

The model \eqref{radeuler} derives from a hyperbolic--kinetic system where the radiation
is described by an additional variable for the photons density satisfying a transport equation
with interaction kernel given by the Stefan--Boltzmann law.
System \eqref{radeuler} is obtained in the non-relativistic limit (speed of light tending to $+\infty$)
and the variable $n$ is an average of the photon density.
The elliptic equation for the variable $n$ emerges from the assumption that the 
variable $n$ is linked with the temperature by the integral relation
\[
	n(x,t)=\int_{\mathbb{R}} g(\theta(y,t))\,K(|x-y|)\,dy
\]
with an exponential kernel $K(s)=\frac12\,\tau\,e^{-\tau\,s}$.
Details on the derivation of asymptotic regimes as the one described by \eqref{radeuler}
can be found in \cite{BuetDesp04, GodiGoud05, LowrMoreHitt99}.

As usual, the pressure is considered a function of density and temperature, $p=p(\rho,\theta)$ and
the specific total energy $E$ is given by
\[
	E=\frac12\,u^2+e
\]
where the internal energy $e$, in general, may depend on $\rho$ and $\theta$, $e=e(\rho,\theta)$.
Precisely, we concentrate on the case of polytropic perfect gases; 
this amounts in considering the functions $p, e, g$ given by the formulas
\begin{equation}\label{specialpeg}
	p=R\,\rho\,\theta,\qquad
	e=\frac{R}{\gamma-1}\,\theta,\qquad
	g(\theta)=\sigma\,\theta^\alpha
\end{equation}
where $\gamma>1$ is the adiabatic constant ($\gamma=5/3$ for monoatomic gases,
and $\gamma=7/5$ for diatomic gases), $R>0$ the perfect gas constant,
$\sigma$ the Stefan--Boltzmann constant and $\alpha=4$.

Existence of shock profiles in the context of radiation hydrodynamics have been considered
for many years.
In \cite{HeasBald63, ZeldRaiz02, MihaMiha84}, a formal analysis is performed concentrating
on specific ranges for the asymptotic states relative to the size of the shock and the strength
of the radiation effect.
More recently, rigorous proof has been presented:\\
-- for a simplified model, known as the  {\sf Hamer model}, consisting of a single
scalar conservation law coupled with an elliptic equation, mimicking the form of the system
\eqref{radeuler} (small and smooth shocks \cite{SchoTadm92}, 
general and possibly discontinuous shocks \cite{KawaNish99, LMS1, LMS2});\\
-- for systems, both for the special form \eqref{radeuler} (small and smooth shocks \cite{LinCoulGoud06},
large and discontinuous shocks \cite{CGLL}) and for general hyperbolic--elliptic systems
(small shocks and linear coupling \cite{LMS1}, small shocks and nonlinear coupling \cite{LMS2}).
\vskip.25cm

Here, we contribute to the exploration of system \eqref{radeuler} showing existence 
of shock profiles for all of the possible regimes: small, medium and large size.
Apart for the requests on the form of constutive functions $p$, $e$ and $g$, the unique
assumption is that the asymptotic states connected by the shock profile determine an 
admissible shock wave for the corresponding reduced hyperbolic system.

\begin{theorem}\label{thm:existeuler}
Assume \eqref{specialpeg} with $R>0, \gamma>1, \alpha>0$.
Let the triples $(\rho_\pm, u_\pm, \theta_\pm)$ and the constant $c\in\mathbb{R}$ 
be such that the Rankine--Hugoniot conditions
\[
	c[\rho]=[\rho\,u],\qquad	c[\rho\,u]=[\rho\,U^2+p],\qquad		c[\rho\,E]=[\rho\,E\,U+p\,U]
\]
are satisfied, together with the entropy condition for 1-shocks
\[
		c+\sqrt{\gamma\,\theta_-}<u_-, \qquad
		c<u_+<c+\sqrt{\gamma\,\theta_+}
\]
Then system \eqref{radeuler} supports a shock wave, that is solution of the form
\[
	(\rho, u, \theta, n)=(\rho, u, \theta, n)(x-ct)
\]
such that $(\rho,u,\theta,n,\partial_x n)(\pm\infty)=(\rho_\pm,u_\pm,\theta_\pm,g(\theta_\pm),0)$.

The profile is unique up to translations and possesses at most one discontinuity point; 
if this is the case, all of the variables $\rho, u$ and $\theta$ have a jump at such point.
\end{theorem}

Once the existence of the shock waves is estabilished, it is meaningful to determine
qualitative properties of the profile and, precisely, its smoothness and its monotonicity.
To this aim, we introduce the velocities $U_\pm:=u_\pm-c$, giving the speed of propagation
of the asymptotic states relatively to a reference frame joint with the shock.
The following statement concerns with the regularity of the profile, giving sufficient conditions
for the presence of a jump point, and with the monotonicity of the component of the 
profile. 
Both they are stated in term of the key parameter given by the ratio between the velocity jump
of the shock,  i.e. $|u_+-u_-|=|U_+-U_-|$, and the average velocity given by the arithmetic mean
$(U_++U_-)/2$.
Precisely, we consider the parameter 
\begin{equation*}
	\delta:=\frac{|U_+-U_-|}{U_c}\qquad\textrm{where}\quad U_c:=\frac{U_++U_-}{2},
\end{equation*}
so that the measure of the shock size is always normalized with respect to the average
speed of propagation of the asymptotic states relatively to the shock velocity.

\begin{theorem}\label{thm:qualitative}
Assume \eqref{specialpeg} with $R>0, \gamma>1, \alpha>0$.\\
{\bf i.} If $\gamma<2$, the profile of the shock wave determined in Theorem \ref{thm:existeuler}
is discontinuous in the components $\rho, u$ and $\theta$ if 
\[
	\delta\geq \delta_{\textrm{jump}}:=\frac{2(\gamma-1)}{\gamma}\,.
\]
{\bf ii.} The component $\rho$ and $u$ of the profile of the shock determined
in Theorem \ref{thm:existeuler} are always monotone, while the  temperature
profile is monotone if and only if $\gamma<3$ and 
\[
	\delta\leq \delta_{\textrm{spike}}:=\frac{\gamma-1}{\gamma}\,.
\]
If this condition is not satisfied, the temperature profile has a single internal
absolute maximum point that can be either attained at the jump point 
or at some point of regularity.
\end{theorem}

As a consequence of the entropy conditions, the value of $\delta$ belong always to the interval
$(0,2/\gamma)$.
For $\gamma\in(1,3)$, there holds $0<\delta_{\textrm{spike}}<\delta_{\textrm{jump}}<2/\gamma$,
hence, depending on the range of values of the parameter $\delta$, we can distinguish three
different type of temperature profiles: monotone and continuous, non-monotone and continuous,
non-monotone and discontinuous. 

The smallness of the parameter $\delta$ is not equivalent to the one 
of the shock stength $|(\rho_+, U_+, \theta_+)-(\rho_-, U_-, \theta_-)|$.
Indeed, from the Rankine--Hugoniot relations, we infer
\[
	[\rho]=\frac{4\,\rho_\pm U_\pm}{4-\delta^2}\,\frac{\delta^2}{[u]},
		\qquad
	[\theta]=\frac{\gamma-1}{\gamma\,R}\,\frac{[u]^2}{\delta}
\]
and thus the small shock assumption corresponds to the requests 
\[
	[u]^2\ll \delta\ll |[u]|^{1/2}\ll 1.
\]
As showed in \cite{LinCoulGoud06, LMS2}, in this limiting regime, the profile is smooth.
Theorem \ref{thm:qualitative} gives the necessary condition $\delta<\delta_{\textrm{jump}}$ 
for such regularity.
An additional condition is also obtained during the proof, see \eqref{condspiral}, and not
reported here because of its scarce readability.
\vskip.25cm

With respect to previous rigorous results on radiative shocks Theorems \ref{thm:existeuler}
and \ref{thm:qualitative} can be considered as conspicous improvements.
By restricting the attention to a specific system, we can implement a strategy closely
resembling the one considered in \cite{LMS1, LMS2} passing from a local study
for small shocks to a global analysis for any possible shock size.
With respect to \cite{CGLL}, where the case of large shocks is considered, in addition
to the removal of the size assumption, we reduce drastically the assumption on the 
data of the problem. 

As usual, the proof is based on the analysis of the system of ordinary differential equations
obtained by looking for traveling wave solutions to \eqref{radeuler}.
Taking advantage of the conserved quantity, such fifth-order system can be reduced to 
a second-order one.
The main original ingredient in the proof resides in the choice of the variable selected for 
the reduction.
The analysis in \cite{CGLL, HeasBald63, KawaNish99, LinCoulGoud06} is based on the
study of the second-order differential equation solved by the velocity variable $u$;
differently, in \cite{LMS1, LMS2}, it is preferred to consider the second-order equation 
for the radiation variable $n$.
Here, we decide for the somewhat intermediate approach given by considering the 
first-order system for velocity $u$ and radiation $n$. 
Such system is transformed in the standard form, that is common to the one valid
for the Hamer model, and the subsequent analysis consists in a detailed study of such
a prototypical system, for which the existence of a heteroclinic orbit, possibly discontinuous
is proved (see Proposition \ref{prop:existence}).
Theorems \ref{thm:existeuler} and \ref{thm:qualitative} descend from Proposition \ref{prop:existence}
and the construction of the heteroclinic orbit for the reduced system.

The paper is organized as follows.
In Section 2, we show how to reduce the original problem for the radiation hydrodynamics
model \eqref{radeuler} to a standard form. 
For completeness, we also present how the same strategy operates in the case of the Hamer model.
Section 3 is devoted to the analysis of the system in standard form and to the construction of 
heteroclinic orbits connecting the two singular points.
In Section 4, we resume to the original model, interpreting the construction of the previous Section
in term of the hydrodynamics variables.
Finally, we draw conclusions in Section 5, pointing at some possible appealing directions
for future research in the subject.

\section{Scaling down to the reduced system}\label{sec:scaling}

Due to the conservation structure, the system of ordinary differential equations
for the profile of the radiative shock under investigation can be reduced to a 
first order system for a two-dimensional variable.
The aim of this Section is to show that such a reduced system has the form 
\begin{equation}\label{std}
		\frac{dx}{d\xi} =\frac{1}{\nu}\,\frac{G(x)-y}{x},\qquad
		\frac{dy}{d\xi}=\frac{1}{2}\,\nu\,(1-x^2).
\end{equation}
for an appropriate function $G$ and positive parameter $\nu$.
The same procedure can be applied to a simpler system given by the coupling 
of a scalar conservation law with an elliptic equation.
For pedagogical reasons, we first present this simpler situation and postpone 
the discussion of the Euler system with radiation \eqref{radeuler}.
Since the two derivations of \eqref{std} are independent, the reader
may skip the first part and move directly to the case of radiation hydrodynamics.

\subsection*{Hamer model}
Given smooth functions $f$ and $g$, let us consider the hyperbolic-elliptic system,
known as {\sf Hamer model} (see \cite{Hame71} and descendants)
\begin{equation}\label{hamer}
	\left\{\begin{array}{l}
		\partial_t u+\partial_x\bigl(f(u)-\partial_x n\bigr)=0,\\
		\partial_x^2 n - n + g(u) = 0.
	\end{array}\right.
\end{equation}
For convex fluxes $f$, the hyperbolic scalar conservation law obtained by disregarding the coupling term
$\partial_x n$  possesses shock waves connecting states $u_\pm$ if and only if $u_+<u_-$.
Such waves have a piecewise constant profile with a single jump point and
propagates with a speed $c$ that is given by the Rankine--Hugoniot condition, {\it viz.}
\begin{equation}\label{rh}
	c=\frac{[f(u)]}{[u]}=\frac{f(u_+)-f(u_-)}{u_+-u_-}
\end{equation}
In analogy, we look for traveling fronts $(u,n)=(u,n)(x-ct)$ that solves \eqref{hamer}
and satisfy the asymptotic conditions
\[
	u(\pm\infty)=u_\pm\qquad n(\pm\infty)=g(u_\pm),\qquad \partial_x n(\pm\infty)=0.
\]
Apart for the requests  $u_+<u_-$ and $c$ given by \eqref{rh}, we also assume
\begin{equation}\label{hypg}
	[g]:=g(u_+)-g(u_-)<0.
\end{equation}
Setting $\xi=x-ct$, the couple $(u,n)$ satisfies an algebraic-differential system
\begin{equation}\label{burgers3d}
		m=f(u)-f(u_\pm)-c(u-u_\pm),\qquad
		\frac{dn}{d\xi}=m,\qquad
		\frac{dm}{d\xi} = n - g(u),
\end{equation}
to be interpreted as a two-dimensional dynamical system along a surface in 
the three-dimensional space $(u,n,m)$.

There are different possible ways to deal with \eqref{burgers3d}, corresponding to
different choices of unknown.
The approach in \cite{SchoTadm92, KawaNish99} consists essentially in deriving a second order
differential equation for the variable $u$, while the strategy used in \cite{LMS1, LMS2}, dictated by the 
better regularity properties of the variable $n$, corresponds to consider the couple $(n,m)$.
Differently, we deal with \eqref{burgers3d} in a way that is somewhat in between the two and we look
for a reduced system for the couple $(u,n)$.
This method avoids the necessity to invert the relation $m=f(u)-f(u_\pm)-c(u-u_\pm)$, but it has the
drawback of the possible presence of jumps of the trajectory.

Differentiating the first equation, we get
\[
		\frac{dm}{d\xi}=\left(f'(u)-c\right)\frac{du}{d\xi},
\]
so that system \eqref{burgers3d} is completely described by the reduced system
\[
		\frac{du}{d\xi} = \frac{n - g(u)}{f'(u)-c},\qquad
		\frac{dn}{d\xi}=f(u)-f(u_\pm)-c(u-u_\pm).
\]
In the case of Burgers equation, i.e. $f(u)=\frac12\,u^2$, $c=(u_++u_-)/2$,
\begin{equation}\label{burgers2d}
		\frac{du}{d\xi} = \frac{n - g(u)}{u-c},\qquad
		\frac{dn}{d\xi}=\frac{1}{2}(u-u_-)(u-u_+).
\end{equation}
Introducing the variable $(x,y)$ related with $(u,n)$ by
\[
	u=\frac{1}{2}\Bigl\{[u]\,x+u_-+u_+\Bigr\}
		\qquad\textrm{and}\qquad 
	n=\frac{1}{2}\Bigl\{[g]\,y+g(u_-)+g(u_+)\Bigr\},
\]
where $[u]=u_+-u_-<0$, system \eqref{burgers2d} takes the form \eqref{std} where
\[
	\nu:=-\frac12\,\frac{[u]^2}{[g]}\qquad\textrm{and}\qquad
	G(x):=\frac{2(g\circ u)(x)-g(u_-)-g(u_+)}{[g]}.
\]
The function $G$ and the parameter $\nu$ have to be regarded as functions of the asymptotic
states $u_\pm$ or, equivalently, of the values $[u]$ and $c$.
By definition, for any choice of $u_\pm$ there hold $G(\pm 1)=\pm 1$.

\begin{example}\label{ex:gburgers}\rm
For a linear term $g$, {\it viz.} $g(u)=\sigma u$ for some $\sigma>0$, the expressions
for $\nu$ and $G$ are
\[
	\nu=-\frac1{2\sigma}\,[u]
	\qquad\textrm{and}\qquad
	G(x)=x.
\]
In the quadratic case, $g(u)=\sigma\,u^{2}$, $\sigma>0$, we obtain
\[
	\nu=-\frac{1}{4\sigma}\,\frac{[u]}{c}\qquad\textrm{and}\qquad
	G(x)=x-\frac{1}{4}\,\frac{[u]}{c}(1-x^2).
\]
where $c=\frac12(u_++u_-)$.
In this case, both $\nu$ and $G$ are homogeneous with respect to the couple $([u],c)$
and thus system \eqref{std} does not vary if the ratio $[u]/c$ is kept fixed.

In general, for $g(u)=\sigma\,u^{\alpha}$ for some exponent $\alpha>0$,
the value $[g]$ is homogeneous of degree $\alpha$ with respect to $[u]$ and $c$,
and, as a consequence, the parameter $\nu$ is a homogeneous function of degree
$2-\alpha$ with respect to the couple $([u],c)$ and $G$ is homogeneous of degree 0.
Explicit expressions for the corresponding $\nu$ and $G$ does not seem
particularly significant.
\end{example}

Solutions to system \eqref{std} may possess eventual discontinuity.
At such point, $y$ is continuous (as a consequence of the continuity of $n$)
and, denoting by $x_\ell$ and $x_r$ the values of $x$ at the left and at the right
of the jump point, respectively, there hold
\[
	x_\ell<x_r
		\qquad\textrm{and}\qquad
	x_\ell+x_r=0.
\] 
Solutions to \eqref{std} with the above jump conditions will be analyzed in the
Section \ref{sec:reduced}.
\vskip.25cm

In the remaining part of the present Section, we show how system \eqref{std} emerges also
when looking for traveling wave solutions fot the radiating Euler system \eqref{radeuler}.

\subsection*{Radiation hydrodynamics}
Next, we consider the hyperbolic-elliptic system \eqref{radeuler}
describing the evolution of a compressible fluid under the effect of radiation.
Setting $\sigma_s^{-1}=0$, the first three equations in system \eqref{radeuler} reduce 
to a standard hyperbolic model for compressible fluids.
Such system possesses shock solutions, i.e. traveling wave solutions $W=W(x-ct)$
with piecewise constant profile $W$ of the form
\[
	W(\xi)=W_-\chi_{{}_{(-\infty,0)}}(\xi)+W_+\chi_{{}_{(0,+\infty)}}(\xi).
\]
for appropriate states $W_\pm$ and parameter $c$.
The shock speed $c$ is related with the states $W_\pm$ by means
of the Rankine--Hugoniot condition.
Precisely, setting $U=u-c$, there hold
\[
	\left\{\begin{aligned}
		&\rho_-\,U_-=\rho_+\,U_+,\\
		&\rho_-\,U_-^2+p_-=\rho_+\,U_+^2+p_+,\\
		&\rho_-\,U_-\left(\frac12\,U_-^2+e_-\right)+p_-\,U_-
			=\rho_+\,U_+\left(\frac12\,U_+^2+e_+\right)+p_+\,U_+
	\end{aligned}\right.
\]
In the case \eqref{specialpeg}, if $\rho_\pm U_\pm\neq 0$, 
the values $\theta_\pm$ satisfy the linear system
\[
	U_+\theta_--U_-\theta_+=\frac{1}{R}\,U_-U_+[u],\qquad
	\theta_--\theta_+=\frac{\gamma-1}{2\,\gamma\,R}\,(U_++U_-)[u],
\]
hence, they can be written as explicit functions of $U_\pm$
\[
	\theta_-=\frac{1}{\gamma\,R}\Bigl(U_c-\frac12[u]\Bigr)\Bigl(U_c+\frac{\gamma}2[u]\Bigr),
		\qquad
	\theta_+=\frac{1}{\gamma\,R}\Bigl(U_c+\frac12[u]\Bigr)\Bigl(U_c-\frac{\gamma}2[u]\Bigr)
\]
where $U_c:=\frac{1}{2}(U_-+U_+)$.
Note also that, if the jump $[u]$ is strictly negative, then the jumps $[\rho]$ and $[\theta]$
are both strictly positive.

Still taking advantage of \eqref{specialpeg}, 
entropy conditions in the case of 1-shocks require $U_\pm, \theta_\pm$ to be such that
\begin{equation}\label{entropy1}
	\sqrt{\gamma\,\theta_-}<U_-,
		\qquad
	0<U_+<\sqrt{\gamma\,\theta_+},
\end{equation}
which implies $U_+<U_-$, since $[\theta]$ and $[u]$ have opposite sign.

For later use, let us set
\begin{equation}\label{ABC}
	A:=\rho_\pm U_\pm, \qquad
	B:=\rho_\pm\,U_\pm^2+p_\pm\qquad
	C:= \rho_\pm\,U_\pm\left(\frac12\,U_\pm^2+e_\pm\right)+p_\pm\,U_\pm.
\end{equation}
or, equivalently, 
\[
	A:=\rho_\pm U_\pm, \qquad
	B:=\frac{\gamma+1}{2\gamma}\,A\,(U_-+U_+)\qquad
	C:=\frac{\gamma+1}{2(\gamma-1)}\,A\,U_-\,U_+.
\]
having used \eqref{specialpeg}.
\vskip.25cm

We are interested in determining the existence and the internal structure of
traveling wave solutions 
\[
	(\rho, u, \theta, n)=(\rho, u, \theta, n)(x-ct)
\]
solutions to \eqref{radeuler} such that
$(\rho,u,\theta,n,\partial_x n)(\pm\infty)=(\rho_\pm,u_\pm,\theta_\pm,g(\theta_\pm),0)$.
Plugging the traveling wave ansatz into the system, we get the system of differential equation
\begin{equation}\label{twode}
	\begin{aligned}
		&\frac{d}{d\xi}\left(\rho\,U\right)=0,	
		&\qquad &\frac{d}{d\xi}\left(\rho\,U^2+p\right)=0,\\
		&\frac{d}{d\xi}\left(\rho\,U\left(\frac12\,U^2+e\right)+p\,U
			-\frac{1}{\sigma_s}\,\frac{dn}{d\xi}\right)=0,
		&\qquad &\frac{d^2 n}{d\xi^2}=\tau^2\bigl(n - g(\theta)\bigr)
	\end{aligned}
\end{equation}
to be satisfied in the regions where the wave profile is smooth.

The conservative form of the first two equations implies that
\[
	\rho\,U=A\qquad\textrm{and}\qquad \rho\,U^2+p=B
\]
all along the trajectory for any weak solution to \eqref{twode}
where $A$ and $B$ are given in \eqref{ABC}.
Thus, there hold
\begin{equation}\label{rhoepidiu}
	\rho=\rho(U)=\frac{A}{U},\qquad p=p(U)=B-A\,U
\end{equation}
and system \eqref{twode} reduces to the algebraic-differential system
\begin{equation}\label{twode2}
		m=\sigma_s f(U),\qquad
		\frac{dn}{d\xi}=m,\qquad 
		\frac{dm}{d\xi}=\tau^2\bigl(n - (g\circ\theta)(U)\bigr)
\end{equation}
where, for $A, B, C$ as in \eqref{ABC},
\begin{equation}\label{generalf}
	f(U):=-\frac12\,A\,U^2+A\,e(U)+B\,U-C,
\end{equation}
and $e(U)=e(\rho(U),\theta(U))$ with $\rho(U)$ and $\theta(U)$ are obtained from 
\eqref{specialpeg} and \eqref{rhoepidiu}.
System \eqref{twode2} describes a two-dimensional dynamical system in the three-dimensional
space $(U,n,m)$ along the manifold $\Sigma$ determined by the algebraic relation $m=\sigma_s f(U)$.
Using the special form for pressure $p$ and internal energy $e$ given in \eqref{specialpeg}, the function $f$
can be rewritten as
\begin{equation}\label{specialf}
 f(U)=-\kappa\,(U-U_+)(U-U_-)
 	\qquad\textrm{where}\quad 
 	\kappa:=\frac{\gamma+1}{2(\gamma-1)}\,A.
\end{equation}
If the profile has a jump point at some point, the function $n$ and its first derivative are forced to be continuous
at such point, since the last equation in \eqref{twode2} forces the second derivative of $n$ to be bounded
and measurable.
As a consequence, denoting by $\rho_\ell, U_\ell, \theta_\ell$ and $\rho_r, U_r, \theta_r$ the values
at the left-/right-hand side of the discontinuity, the relation $f(U_\ell)=f(U_r)$ holds together with the
admissibility conditions that read as, in the case of 1-shocks,
\[
	\sqrt{R\,\gamma\,\theta_\ell}<U_\ell,
		\qquad
	0<U_r<\sqrt{R\,\gamma\,\theta_r}
\]
Because of \eqref{specialf}, jumps are admissible if and only if
\begin{equation}\label{entropyU}
	U_c<U_\ell<2U_c \qquad\textrm{where}\quad
	U_c:=\frac{U_++U_-}{2}
\end{equation}
the corresponding value $U_r$ being given by $2U_c-U_\ell$.
\vskip.25cm

There are different possible ways to deal with \eqref{twode2}, corresponding to
different choices of unknown.
As for the Hamer model, two main approaches have been used:
in \cite{HeasBald63, LinCoulGoud06, CGLL} a second order differential
equation for the variable $U$ is derived, in \cite{LMS1, LMS2} it has been
considered the dynamics for the couple $(n,m)$, because of the greater regularity
of the variable $n$.
Here, we proceed by deriving a system for the couple $(U,n)$, hence, with the possibility
of trajectories that are discontinuous at some point in the first component of the unknowns.

Differentiating with respect to $\xi$ the equation defining the surface $\Sigma$ and eliminating
the variable $m$, we get the reduced two-dimensional system
\begin{equation}\label{redUn}
	\frac{dU}{d\xi} =\frac{\tau^2}{\sigma_s}\,\frac{n-(g\circ\theta)(U)}{f'(U)},\qquad
	\frac{dn}{d\xi} =\sigma_s\,f(U).
\end{equation}
Thanks to \eqref{specialpeg}, function $f$ is the second-order polynomial given in \eqref{specialf}
and  $f'(U)=-2\kappa(U-U_c)$.
The final step consists in introducing new coordinates to transform the system in the
standard form \eqref{std}.
\vskip.25cm

Introducing the variable $(x,y)$ related with $(U,n)$ by
\[
	U=U(x):=\frac{1}{2}\Bigl\{[u]\,x+U_-+U_+\Bigr\}
		\quad\textrm{and}\quad 
	n=n(y):=\frac{1}{2}\Bigl\{[g]\,y+g(\theta_-)+g(\theta_+)\Bigr\},
\]
system \eqref{redUn} becomes
\[
	\frac{dx}{d\xi} =\frac{\tau^2\,[g]}{\sigma_s\,\kappa\,[u]^2}\,\frac{G(x)-y}{x},\qquad
	\frac{dy}{d\xi} =\frac{\sigma_s\,\kappa\,[u]^2}{2[g]}\,\nu(1-x^2)
\]
where
\begin{equation}\label{defG}
	G(x):=\frac{1}{[g]}\bigl(2\,(g\circ\theta\circ U)(x)-g(\theta_-)-g(\theta_+)\bigr)
\end{equation}
By rescaling the variable $\xi$ by setting $\xi=\tau\,\zeta$, we end up with the system \eqref{std} where
\begin{equation}\label{defnu}
	\nu:=\frac{\kappa\,\sigma_s\,[u]^2}{\tau\,[g]}.
\end{equation}
Both function $G$ and parameter $\nu$ depends on the values of the asymptotic states
of the profile and thus can be regarded as functions of the parameters $[u]$ and $U_c$.

Let us also observe that the function $\theta\circ U=(\theta\circ U)(x)$ is given by
\[
	(\theta\circ U)(x)
			=\frac{1}{R}\left(\frac{\gamma+1}{\gamma}\,U_c-U\right)U
		=\frac{1}{\gamma\,R}\left(U_c+\frac12[u]\,x\right)\left(U_c-\frac12\,\gamma\,[u]\,x\right);
\]
thus, there holds
\[
	\frac{d}{dx}(\theta\circ U)(x)
		=-\frac{[u]^2}{2R}\left(x-x_c\right)
	\qquad\textrm{where }
		x_c:=-\frac{\gamma-1}{\gamma}\frac{U_c}{[u]}>0.
\]
As a consequence, if $g$ is strictly increasing, the function $G$ has the same monotonicity
of $\theta\circ U$ and it is strictly increasing in $[-1,1]$ if and only if $x>1$, that is if and only if
\[
	\bigl|u_+-u_-\bigr|\leq \frac{\gamma-1}{\gamma}\,U_c \, .
\]
Otherwise, $G$ is strictly increasing in $[-1,x_c]$ and strictly decreasing in $[x_c,1]$.

\begin{example}\label{ex:geuler}\rm
In the case $g(\theta)=\sigma\,\theta$ with $\sigma>0$, the parameter $\nu$ and 
the function $G$ can be computed explicitly
\[
	\begin{aligned}
	\nu&=\frac{\kappa\,\sigma_s}{\tau\,\sigma}\,\frac{[u]^2}{[\theta]}
		=\frac{\gamma(\gamma+1)\,R}{2(\gamma-1)^2}\,
			\frac{\sigma_s\,A}{\tau\,\sigma}\,\delta,\\
	G(x)&=\frac{2(\theta\circ U)(x)-\theta_+-\theta_-}{\theta_+-\theta_-}
		=x+\frac{\gamma}{2(\gamma-1)}\,\delta(1-x^2)
	\end{aligned}
\]
Thus, the structure is essentially the same of the Hamer model \eqref{hamer}
with a quadratic function $g$ (see Example \ref{ex:gburgers}).

Higher powers in the temperature dependence of function $g$, i.e. $g(\theta)=\sigma\,\theta^\alpha$
with $\sigma>0$ and $\alpha$ a positive integer, give raise to complicate formulas, whose explicit
expression is not particularly significant.
It is relevant to observe that, being $\theta\circ U$ a homogeneous function of degree 2 with respect
to $[u]$ and $U_c$, the value $[g]$ is homogeneous of degree $2\alpha$ and thus $\nu$ is homogeneous
of degree $2(1-\alpha)$ and $G$ is homogeneous of degree $0$ with respect to the same variables.
In particular, when the ratio
\[
	\delta:=-\frac{[u]}{U_c}=\frac12\,\frac{U_--U_+}{U_-+U_+}
\]
is kept fixed the function $G$ does not change and the parameter $\nu$ decreases
as $-[u]$ increases if $\alpha>1$.
\end{example}

An eventual discontinuity in the trajectory keeps $y$ fixed;
moreover, if $x_\ell$ and $x_r$ denote the values at the left and at the right of the jump, 
then $x_r=-x_\ell$ and condition \eqref{entropyU} translates into
\begin{equation}\label{entropyx}
	\frac{U_++U_-}{U_+-U_-}<x_\ell<0.
\end{equation}
Note that the value $(U_++U_-)/(U_+-U_-)$ is always strictly smaller than $-1$,
so that jumps are always alloweded if $x\in[-1,0)$.

\begin{remark}\rm
The analysis is based on the choice \eqref{specialpeg}, that playes a key-r\^ole
in determining the specific expression \eqref{specialf} for the function $f$.
For more general expression of pressure $p$ and internal energye $e$, 
all of the subsequent analysis can still be performed  if the function $f$,
given by formula \eqref{generalf}, is strictly concave.
\end{remark}

\section{Analysis of the reduced system}\label{sec:reduced}

In this Section, we consider piecewise smooth solutions to the system
\begin{equation}\label{std2}
		\frac{dx}{d\zeta} =\frac{1}{\nu}\,\frac{G(x)-y}{x},\qquad
		\frac{dy}{d\zeta}=\frac{1}{2}\,\nu\,(1-x^2).
\end{equation}
with $\nu>0$ and $G$ smooth and such that $G(\pm 1)=\pm 1$.
At any eventual jump point $\xi$, we assume that the couple $(x,y)$
possesses left/right limits, denoted by $(x_\ell, y_\ell)$ and $(x_r,y_r)$,
and that 
\[
	x_\ell<x_r,\qquad 
	x_\ell+x_r=0,\qquad
	y_\ell-y_r=0
\]
System \eqref{std2} has two singular points, $P_-:=(-1,-1)$ and $P_+:=(+1,+1)$, and
we are interested in  the existence of a heteroclinic orbit connecting such that 
$(x,y)(\pm\infty)=P_\pm$.
Since both critical points $P_\pm$ are saddles, such orbit is obtained by matching the
unstable one-dimensional manifold of $P_-$ with the stable one-dimensional manifold
of $P_+$ (see Lemma \ref{lem:stabunstab}).

\begin{figure}[ht]\label{fig:pianofasi}
\includegraphics[width=8cm]{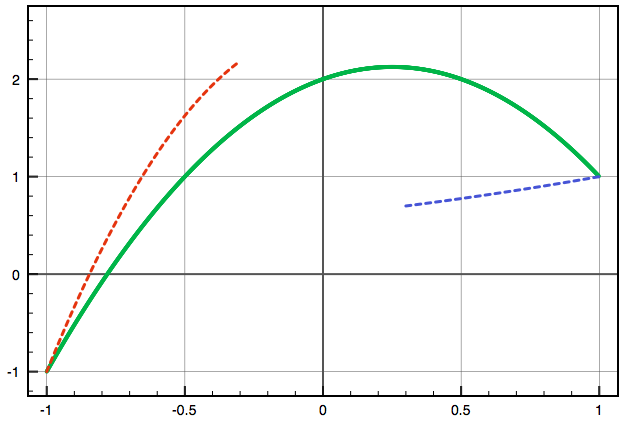}
\caption{\footnotesize The phase plane $(x,y)$ with the graph of function $G$ (continuous line) 
and the unstable/stable manifold of $P_-=(-1,-1)/P_+=(+1,+1)$  (dashed lines).}
\end{figure}

The matching may happen either in a continuous or in a discontinuous way, the former
case being related with the nature of the system at the singular line $x=0$
(see Lemma \ref{lem:critical}).

\begin{lemma}\label{lem:stabunstab}
For any $\nu>0$, the critical points $P_\pm$ are saddles and, denoted by $\mathcal{U}_-$ 
the part of unstable manifold of $P_-$ in $\{x>-1\}$ and by $\mathcal{S}_+$ the part of the 
stable manifold of $P_+$  in $\{x<1\}$,  
there exist functions $\phi_-=\phi_-(x)$ defined in $[-1,0]$ and  $\phi_+=\phi_+(x)$ defined in $[0,1]$
both differentiable and strictly monotone increasing, such that
\begin{align*}
	&\{(x,y)\in\mathcal{U}_-\,:\,x\in [-1,0)\}=\{(x,y)\,:\,x\in [-1,0),\, y=\phi_-(x)\}\\
	&\{(x,y)\in\mathcal{S}_+\,:\,x\in (0,+1]\}=\{(x,y)\,:\,x\in (0,+1],\, y=\phi_+(x)\}
\end{align*}
Moreover, $\phi_+(x)<G(x)$ for any $x\in(-1,0)$ and $\phi_-(x)>G(x)$ for any $x\in(0,1)$.
\end{lemma}

\begin{proof}
The jacobian matrices of \eqref{std2} at $P_\pm$
\[
	\frac{1}{\nu}\begin{pmatrix}
		\pm G'(\pm 1)	&	\mp 1 \\
		\mp \nu^2				&	0
	\end{pmatrix},
\]
have both determinant equal to $-1$; hence, both points are saddles.

The manifolds $\mathcal U_-$ and $\mathcal S_+$ have tangent directions given by
the vectors $(\mu_-^u,\nu)$ and $(-\mu_+^s,\nu)$, respectively, where
\[
	\mu_-^u:=\frac{2\nu}{G'(-1)+\sqrt{G'(-1)^2+4\nu^2}},
	\qquad
	\mu_+^s:=-\frac{2\nu}{G'(+1)+\sqrt{G'(+1)^2+4\nu^2}},
\]
Thus, if $G'(-1)\leq 0$, $\mathcal U_-$ lies in the region $\{(x,y)\,:\,y>G(x)\}$
in a neighborhood of $P_-$;
similarly,  if $G'(1)\leq 0$, $\mathcal S_+$ lies in the region $\{(x,y)\,:\,y<G(x)\}$
in a neighborhood of $P_+$.

If $G'(\pm 1)>0$, since $\mu_-^u<\nu/G'(-1)$ and $-\mu_+^s<\nu/G'(+1)$,
for $x\in(-1,1)$, the manifold $\mathcal U_-$ lies in the region $\{(x,y)\,:\,y>G(x)\}$
and the manifold $\mathcal S_+$ in the region $\{(x,y)\,:\,y<G(x)\}$ in a neighborhood
of $P_-$ and $P_+$, respectively.
Since $\Omega_-:=\{(x,y)\,:\,x\in(-1,0), y>G(x)\}$ is positevely invariant and 
$\Omega_+:=\{(x,y)\,:\,x\in(0,1), y<G(x)\}$
is negatively invariant, the $\mathcal U_-$ is described in $\Omega_-$ by a curve $y=\phi_-(x)$ 
and the manifold $\mathcal S_+$ is described in $\Omega_+$ by a curve $y=\phi_+(x)$.
Moreover, both functions are monotone increasing with respect to $x$ since the corresponding
derivatives, given by
\begin{equation}\label{derphi}
	\frac{d}{dx}\phi_\pm(x)=\frac{\nu^2\,x(1-x^2)}{2(G(\phi)-x)},
\end{equation}
do not vanish in the interior of $\Omega_-$ and $\Omega_+$ respectively.
Moreover, because of \eqref{derphi}, the functions $\phi_\pm$ cannot have vertical asymptotes
in the regions $(-1,0)$ and $(0,1)$, respectively;
therefore, they are well defined in $(-1,0)$ and $(0,1)$, respectively.
\end{proof}

For a better understanding of the behavior of the curves $\phi_\pm$ at $x=0$,
we consider the half plane $x>0$ and rescale the independent variable by setting
\[
	x(\zeta)\frac{d}{d\zeta}=\frac{d}{d\eta}
\]
so that \eqref{std2} takes the form
\begin{equation}\label{stdResc}
		\frac{dx}{d\eta} =\frac{1}{\nu}\,\bigl(G(x)-y\bigr),\qquad
		\frac{dy}{d\eta}=\frac{1}{2}\,\nu\,x(1-x^2).
\end{equation}

\begin{lemma}\label{lem:critical}
Assume $G'(0)\neq 0, 2\nu^2$.
The critical point $P_0=(0,G(0))$ of system \eqref{stdResc} is an attractive spiral if $G'(0)<0$,
a repulsive spiral if $0<G'(0)<2\nu^2$ and a source if $2\nu^2<G'(0)$.
In the latter case, there exists $x_+\in(0,1]$ and a monotone increasing function 
$\psi_+=\psi_+(x)$ defined in $[0,x_+]$ such that the fast unstable trajectory of $P_0$ in the region
$\Omega_+=\{x\in(0,1),\,y<G(x)\}$ is given by $(x,\psi_+(x))$.
\end{lemma}

\begin{proof}
The jacobian matrix at the singular point $P_0:=(0,G(0))$
\[
	\frac{1}{\nu}\begin{pmatrix} G'(0) & -1 \\ \nu^2/2	& 0 \end{pmatrix},
\]
has real eigenvalues if and only if $G'(0)^2-2\nu^2>0$.
In the opposite case, both eigenvalues have real parts of the same sign of $G'(0)$.
If $2\nu^2<G'(0)^2$, the eigendirections are given by $(2\mu_0^\pm,\nu)$
where
\[
	\mu_0^\pm=\frac{G'(0)\pm\sqrt{G'(0)^2-2\nu^2}}{2\nu}
\]
Since
\[
	0<\frac{\nu}{2\mu_0^+}<\frac{\nu}{2\mu_0^-}<G'(0)
\]
the principal directions of the critical point $(0,G(0))$ both lies into the region $y<G(x)$ for $x>0$.
The fast unstable direction, corresponding to the eigenvalue $\mu_0^+$, is uniquely determined
and, while lying in $\Omega_+$, is the graph of a monotone increasing function $\psi_+=\psi_+(x)$.
\end{proof}

\begin{remark}\label{rem:rotated}\rm 
If the function $G$ does not depend on the parameter $\nu$ the vector field in
\eqref{stdResc} is a {\sf rotated vector field} in the sense of Duff \cite{Duff53}
in the regions $\Omega_\pm$, as a consequence of the equality
\[
	\det \begin{pmatrix} F & G \\ \partial_\nu F & \partial_\nu G \end{pmatrix}
		=\frac{1}{\nu}\,(G(x)-y)\,x(1-x^2)
\]
(for the analysis of rotated vector field, see also \cite{Perk93}).
Thus, the graphs of the function $\psi_+$ rotates anticlockwise as $\nu$ increase
and the map $\nu\mapsto x^\nu_+$ is monotone decreasing with respect to $\nu$.
Moreover, there exists a critical value $\bar \nu_+$ such that $x^\nu_+=1$ for $\nu<\bar \nu_+$
and $x^\nu_+<1$ for $\nu>\bar \nu_+$.

From Examples \ref{ex:gburgers} and \ref{ex:geuler}, it is readily seen that the eventuality
of a function $G$ independent on $\nu$ is rarely verified in concrete cases.
Nevertheless, in the case of radiation hydrodynamics, the function $G$ does not vary
if the ratio $\delta=-[u]/U_c$ is kept fixed. 
Thus the properties of the rotated vector fields return to be useful in the analysis of variations
of the heteroclinic orbit connecting $P_-$ and $P_+$ when varying the asymptotic states $U_\pm$
keeping $\delta$ fixed.
\end{remark}

A result analogous to Lemma \ref{lem:critical} holds on the side $x<0$, the main difference being related
to the fact that the rescaling of the independent variable changes the orientation of the trajectories.
In particular, we can state that, for $2\nu^2<G'(0)$, there exists $x_-\in[-1,0)$ and a monotone increasing
function $\psi_-=\psi_-(x)$ defined in $[x_-,0]$ such that the fast stable trajectory of $P_0$ in the
region $\Omega_-=\{x\in(0,1),\,y<G(x)\}$ is given by $(x,\psi_-(x))$.

\begin{proposition}\label{prop:existence}
Let $\nu>0$ and $G\in C^1([-1,1])$ be such that $G(\pm 1)=0$.
Then there exists a heteroclinic orbit connecting $P_-$ with $P_+$.
Such orbit possesses at most a single jump, and such discontinuity is 
actually present if one of the following conditions is satisfied
\begin{equation}\label{jump}
	G(0)\leq -1,\quad\textrm{or}\quad
	1\leq G(0),\quad\textrm{or}\quad
	G'(0)<2\nu^2.
\end{equation}
Finally, if $G'(0)>0$ the trajectory of the heteroclinic orbit is unique.
\end{proposition}

\begin{proof}
1. The trajectories $(x,\phi_\pm(x))$, defined in Lemma \eqref{lem:stabunstab}, converge to
the $y-$axis at a finite value of $\zeta$.
Indeed, there holds
\[
	\zeta-\zeta_0=\int_{x(\zeta_0)}^{0} \frac{dx}{dx/d\zeta}
		=\nu\int_{x(\zeta_0)}^{0} \frac{x}{G(x)-\phi_\pm(x)}\,dx.
\]
The latter integral is finite if $\phi_\pm(0)\neq G(0)$.
If $\phi_\pm(0)=G(0)$, then $\phi_\pm'(0)\neq G'(0)$ as can be seen from the 
explicit expression of the principal directions of the singular point $P_0$ of the rescaled 
system \eqref{stdResc} (see Lemma \ref{lem:critical}).
Thus, the function $x/(G(x)-\phi_\pm(x))$ is integrable in small neighborhoods of $0$.\par
2. If $\phi_-(0)=\phi_+(0)=G(0)$, the conjunction of the two orbits $\phi_-$ and $\phi_+$
furnishes the heteroclinic orbit connecting $P_-$ and $P_+$.
In this case, the profile of the corresponding traveling wave is continuous. \par
3. To complete the proof of the first part of the statement, we assume that either
$\phi_-(0)\neq G(0)$ or $\phi_+(0)\neq G(0)$.
Denoted by $\psi_\pm(y)$ the inverse functions of $\phi_\pm$ defined in $[-1,\phi_-(0)]$
and $[\phi_+(0),1]$, respectively, let us consider the function
\[
	h(y):=\psi_+(y)+\psi_-(y)\qquad y\in[\phi_+(0),\phi_-(0)].
\]
Then, by the properties of $\phi_\pm$, the function $h$ is strictly increasing and 
\[
	h(\phi_+(0)):=\psi_-(\phi_+(0))<0<\psi_+(\phi_-(0))=h(\phi_-(0)).
\]
Hence, by continuity of $h$, there exists a single value $y_c$ such that $h(y_c)=0$.
The connection of the trajectory $\phi_-$, truncated at $x=\psi_-(y_c)$,
and $\phi_+$, truncated at $x=\psi_+(y_c)$ gives the desired heteroclinic orbit.\par
4. Any of the conditions in \eqref{jump} implies that one of the trajectories $\phi_\pm$
does not pass through the point $(0,G(0))$.
Indeed, if $G(0)\leq -1$, since the function $\phi_-$ is monotone increasing, then
$\phi_-(0)>-1\geq G(0)$.
Similarly, if $1\leq G(0)$, then $\phi_+(0)$ is strictly greater than $G(0)$.
Finally, if $G'(0)<2\nu^2$, the point $P_0$ is a spiral for the rescaled system \eqref{stdResc}
and no monotone trajectory is alloweded to get onto it.\par
5. To complete the proof, it is sufficient to observe that more can one jump may appear
only if there are smooth solutions of \eqref{std2} passing from the positive side $\{x>0\}$ to
the negative side $\{x<0\}$. 
Because of the structure of the vector field defining the system, the corresponding eventual 
trajectories should pass through the point $P_0=(0,G(0))$, but, in the case $G'(0)>0$, the singular
point $P_0$ is repulsive on $\{x>0\}$  and attractive on $\{x<0\}$ so that no transition from
positive to negative values of $x$ is possible.
\end{proof}

\begin{remark}\rm
In the case $G'(0)\leq 0$, the system \eqref{std2} possesses trajectories passing through
the point $(0,G(0))$ and going from the region $x>0$ to the region $x<0$.
Thus, in principle, solutions may have more than a single jump point and we cannot exclude
that there exists also other possible orbits with the same asymptotic states.
In the concrete case of application of Proposition \ref{prop:existence} to \eqref{hamer}
and \eqref{radeuler} the condition $G'(0)>0$ is satisfied.
\end{remark}

\section{Back to radiation hydrodynamics}\label{sect:zeld}

Next, let us analyze the system \eqref{std2} and the properties of the corresponding heteroclinic orbits
in the case of radiation hydrodynamics as derived in Section \ref{sec:scaling}.
For the reader convenience, let us recall the formulas 
\[
	G(x):=\frac{1}{[g]}\bigl(2\,(g\circ\theta\circ U)(x)-g(\theta_-)-g(\theta_+)\bigr)
		\qquad\textrm{and}\qquad
	\nu:=\frac{\kappa\,\sigma_s\,[u]^2}{\tau\,[g]}.
\]
(see formulas \eqref{defG}--\eqref{defnu}).
As stated in the Introduction, we introduce the key-parameter $\delta:=-[u]/U_c$, so that
the expression of the compound function $(\theta\circ U)=(\theta\circ U)(x)$ becomes
\begin{equation}\label{thetadix}
	(\theta\circ U)(x)=\frac{U_c^2}{\gamma\,R}\left(1-\frac12\,\delta\,x\right)
				\left(1+\frac12\,\gamma\,\delta\,x\right),
\end{equation}
that describes the values of temperature as a function of the auxiliary variable $x$.

We concentrate on a series of significant issues:\\
\indent -- to determine threshold for the presence/absence of disconinuities;\\
\indent -- to analyze the monotonicity of the temperature profiles;\\
\indent -- to study the regimes $\delta\to 0$ and $\delta\to 2/\gamma$;\\
\indent -- to show a short collection of numerical experiments.

At the end of the Section, for completeness, we also mention how the properties
of the reduced system \eqref{std2} translate in the case of the Hamer model.

\subsection*{Discontinuous profiles}
As stated in Proposition \ref{prop:existence}, profiles of the traveling waves are discontinuous
if condition \eqref{jump} is satisfied.
If $g$ is monotone increasing and $[g]>0$, there hold $G(0)\leq -1$ if and only if 
$(\theta\circ U)(0) \leq \theta_-$ and $G(0)\geq 1$ if and only if $\theta_+ \leq (\theta\circ U)(0)$.
Taking advantage of the expression \eqref{thetadix}, the previous relations can be rewritten as 
\begin{align*}
	G(0)\leq -1\;\iff\; 1\leq \left(1+\frac12\,\delta\right)\left(1-\frac12\,\gamma\,\delta\right)
		\;\iff\; \delta\leq \frac{2(1-\gamma)}{\gamma},\\
	G(0)\geq 1\;\iff\; \left(1-\frac12\,\delta\right)\left(1+\frac12\,\gamma\,\delta\right)\leq 1
		\;\iff\; \delta\geq \frac{2(\gamma-1)}{\gamma}.
\end{align*}
For $\gamma>1$ and $\delta>0$, the first condition is never satisfied.

Additionally, there holds
\[
	G'(0)=-\frac{\gamma-1}{R\gamma}\,\frac{[u]}{[g]}\,U_c
		\frac{dg}{d\theta}\left(\frac{U_c^2}{\gamma\,R}\right)
\]
so that condition $G'(0)<2\nu^2$ is equivalent to 
\begin{equation}\label{condspiral}
	\frac{M_0}{A^2}\,U_c\,\frac{dg}{d\theta}\left(\frac{U_c^2}{\gamma\,R}\right)<-\frac{[u]^3}{[g]}
		\qquad\textrm{where}\quad
		M_0:=\frac{2(\gamma-1)^3}{R\gamma(\gamma+1)^2}\,\frac{\tau^2}{\sigma_s^2}
\end{equation}
For $g(\theta)=\sigma\,\theta$, condition \eqref{condspiral} becomes
\[
	\delta>\frac{m_1}{A}
		\qquad\textrm{where}\quad
		m_1:=\frac{\sqrt{2}(\gamma-1)^2}{R\gamma(\gamma+1)}
		\,\frac{\sigma\,\tau}{\sigma_s}
\]
For $g$ as in \eqref{specialpeg} with $\alpha$ integer greater than 1, the explicit formulas 
corresponding to the condition \eqref{condspiral} become very intricated and they do not 
seem to be significant.
In any case, let us stress that differently with respect to the condition $G(0)\geq 1$,
that depends only on the relation between the parameter $\delta$ and the value of the constant $\gamma$,
constraint $G'(0)<2\nu^2$ depends also on the other characteristic parameters of the system.

Summarizing, we can state that the heteroclinic orbit is discontinuous if 
\begin{equation}\label{jumpcond}
	\delta\geq \delta_{\textrm{jump}}:=\frac{2(\gamma-1)}{\gamma}
\end{equation}
Note that, if $\gamma>2$, the region in the plane $(U_c,|[u]|)$ for which this inequality
is satisfied does not intersect the attainable region $\delta<2/\gamma$.


\subsection*{Temperature spikes}
The temperature profile is given by \eqref{thetadix} calculated at $x=x(\zeta)$, the first
component of the heteroclinic orbit of the reduced system \eqref{std2}.
Since $x=x(\zeta)$ is strictly increasing, the monotonicity of the temperature is controlled
by the sign of its first derivative
\[
	\frac{d}{dx}(\theta\circ U)(x)=-\frac{[u]^2}{2\,R}(x-x_c)
	\qquad\textrm{where}\quad
	x_c:=\frac{\gamma-1}{\gamma}\,\frac{1}{\delta}
\]
Hence, the profile has a change in monotonicity  if and only if 
\begin{equation}\label{changemon}
	\delta>\delta_{\textrm{spike}}:=\frac{\gamma-1}{\gamma}
\end{equation}
Such absolute maximum for the temperature is sometime referred to as the {\sf Zel'dovich spike}.
Note that the condition \eqref{changemon} is optimal and depends only on the value of the
constant $\gamma$ (in particular, it does not depend on the specific form of the function $g$!).
For $\gamma>3$, the condition \eqref{changemon} is not satisfied if $\delta<2/\gamma$,
meaning that for such values of $\gamma$ the temperature profile is always monotone increasing.

The location of the absolute maximum point for $\theta$ can be either at the jump
point, if present, or in a region of regularity of the profile. 
In the latter case, the maximum of the temperature is attained at 
$x=x_c$ and its value is
\begin{equation}\label{estimate}
	\theta_{\textrm{\tiny max}}:=(\theta\circ U)(x_c)
		=\frac{(\gamma+1)^2}{4\,\gamma^2}\,\frac{U_c^2}{R}
\end{equation}
If the point of change of monotonicity coincides with the jump point, the values
$\theta_{\textrm{\tiny max}}$ gives only an estimate from above of the maximum
values for the temperature.
It is particularly relevant to note that the expression of $\theta_{\textrm{\tiny max}}$
does depend only on $\gamma$, $R$ and $U_c$.

Estimate \eqref{estimate} can be compared with the corresponding one determined
in \cite{MihaMiha84} (formula (104.67), p.573) for $\gamma<3$
\[
	\theta_{\textrm{\tiny M}}:=(3-\gamma)\,\theta_+.
\]
By using the expression for $\theta_+$ in terms of $U_c$ and $[u]$, we obtain
\[
	\frac{\theta_{\textrm{\tiny max}}}{\theta_{\textrm{\tiny M}}}
	=\frac{(\gamma+1)^2}{4\gamma(3-\gamma)}\,\frac{1}{(1-\delta/2)(1+\gamma\,\delta/2)}.
\]
Since $\delta\in(0,2/\gamma)$, we infer
\begin{equation}\label{compare}
	\frac{1}{3-\gamma}\leq \frac{\theta_{\textrm{\tiny max}}}{\theta_{\textrm{\tiny M}}}
	\leq \frac{(\gamma+1)^2}{8(3-\gamma)(\gamma-1)}.
\end{equation}
In the significant cases $\gamma=5/3$ and $\gamma=7/5$, the above estimates 
becomes
\[
	\gamma=\frac{5}{3}\;:\quad 
	\frac{3}{4}\leq \frac{\theta_{\textrm{\tiny max}}}{\theta_{\textrm{\tiny M}}}\leq 1
		\qquad\qquad
	\gamma=\frac{7}{5}\;:\quad 
	\frac{5}{8}\leq \frac{\theta_{\textrm{\tiny max}}}{\theta_{\textrm{\tiny M}}}
	\leq \frac{9}{8}.
\]
The expression \eqref{compare} can be rewritten as
\[
	1-\frac{2-\gamma}{3-\gamma}\leq \frac{\theta_{\textrm{\tiny max}}}{\theta_{\textrm{\tiny M}}}
	\leq 1+\frac{(3\gamma-5)^2}{8(3-\gamma)(\gamma-1)},
\]
thus, in the regime $\gamma<2$, the first term in this equality chains is strictly smaller than 1
and the last term, apart for the (physical!) case $\gamma=5/3$, 
the estimate from above is strictly greater than $1$.

\subsection*{Limiting regimes}
The expression \eqref{thetadix} shows that $\theta\circ U$ has the form $C\,p(x;\gamma,\delta)$
with $C=U_c^2/\gamma\,R$ and $p$ a polynomial of degree 2 in $x$ with coefficients depending
only on $\delta$ and $\gamma$.
Therefore, in the case of functions $g$ with the power-law form described in \eqref{specialpeg},
the function $G$ in the reduced system \eqref{std} does not changes if the parameters $\gamma$
and $\delta$ are kept fixed.
Here, we want to analyze the behavior of the  heteroclinic orbits built in Section \ref{sec:reduced}
for $\delta$ prescribed in the regimes $\delta\to 0$ and $\delta\to 2/\gamma$.
To this aim, we need to investigate the function $G$ for such values of $\delta$ and 
analyze the variations for the parameter $\nu$.

As $\delta\to 0$, formula \eqref{thetadix} can be rewritten as
\[
	(\theta\circ U)(x)=\frac{U_c^2}{\gamma\,R}\left(1+\frac12(\gamma-1)\,\delta\,x\right)+o(\delta).
\]
and substituting in the definition of $G$, we infer 
\[
	G(x)=x+o(1)\qquad\delta\to 0.
\]
Thus, the function $G$ is strictly increasing and the temperature profile is monotone increasing.
As $\delta\to 2/\gamma$, the relation \eqref{thetadix} becomes
\[
	(\theta\circ U)(x)=\frac{U_c^2}{\gamma^2\,R}\left(\gamma-x\right)\left(1+x\right)
					+o(\delta-2/\gamma)
\]
Thus, we have
\[
	G(x)=G_{\gamma,\alpha}(x)+o(\delta-2/\gamma)\qquad\delta\to \frac{2}{\gamma},
\]
where
\[
	G_{\gamma,\alpha}(x):=-1+2\,\left(\frac{\gamma-x}{\gamma-1}\,\cdot\,\frac{1+x}{2}\right)^{\alpha}
\]
In this case, the function $G$  changes its monotonicity if and only if $1<\gamma<3$ and, if this is the case, 
has an absolute maximum point at $(\gamma-1)/2$; moreover, $G(0)>1$ if and only if $\gamma<2$.
In particular, the temperature profile exhibits the Zel'dovich spike.

Next, we examine the parameter $\nu$ and its behavior for $\delta$ fixed and 
either $[u]\to 0$ or $[u]\to-\infty$ (equivalently, either $U_c\to 0$ or $U_c\to+\infty$). 
The assumption $g(\theta)=\sigma\,\theta^\alpha$ carries also the identity
\[
	[g]=\sigma(\theta_+-\theta_-)(\theta_+^{\alpha-1}+\dots+\theta_-^{\alpha-1})
		=-C\,\sigma\,[u]\,U_c^{2\alpha-1}
\]
for some positive constant $C$ only on $\gamma, \delta, R$ and $\alpha$.
Inserting in the definition of $\nu$ given in \eqref{defnu}, we end up with
\[
	\nu=\frac{C\,\delta}{U_c^{2(\alpha-1)}}
		=\frac{C\,\delta^{2\alpha-1}}{[u]^{2(\alpha-1)}}
\]
for some positive constant $C$ dependent only on 
$\gamma, A=\rho_\pm\,U_\pm,  \sigma, \sigma_a, \sigma_s$.
Hence, for $\alpha>1$, $\nu\to+\infty$ as $[u]\to 0$ (or $U_c\to 0$) 
and $\nu\to 0$ as $[u]\to -\infty$ (or $U_c\to +\infty$).
The case $\alpha=1$ is different, since $\nu$ turns to be constant for $\delta$ fixed
(see Example \ref{ex:geuler}).
For this reason, we only consider the case $\alpha>1$.

As $[u]\to+\infty$, in an appropriate scale, system \eqref{std} formally reduces to
\[
		\frac{dx}{d\zeta} = \frac{G(x)-y}{x},\qquad
		\frac{dy}{d\zeta}=0
\]
For $\delta\to 0$, such structure forces the unstable manifold $\mathcal{U}_-$
of $P_-=(-1,-1)$ and the stable manifold $\mathcal{S}_+$ of $P_-=(+1,+1)$ 
coincide with the graph of the function $G(x)=x+o(\delta)$.
The corresponding temperature profile is continuous and strictly monotone increasing.

For $\delta\to 2/\gamma$, and $\gamma<2$, then the unstable manifold $\mathcal{U}_-$
of $P_-=(-1,-1)$ coincides with the graph of the function $G_{\gamma,\alpha}$;
the stable manifold $\mathcal{S}_+$ of $P_-=(+1,+1)$ is given by a horizontal segment
at height $y=1$.
The heteroclinic orbit is obtained by the connection of such manifold at the (unique) 
point where the stable manifold $\mathcal{U}_-$ is at height $1$ 
(see Figure \ref{fig:casilimite}, left).
Since $G_{\gamma,\alpha}$ is symmetric with respect to $x_c=\frac{1}{2}(\gamma-1)$,
the value $x_-$ of the jump is equal to $2x_c-1=-(2-\gamma)$ and the corresponding
value $x_+$ is $2-\gamma$.
The maximum of the temperature profile is attained at the jump point if and only
if $x_+$ is greater than or equal to $x_c$, that is if and only if $\gamma\leq 5/3$.
\begin{figure}[ht]
\includegraphics[width=7cm, height=5cm]{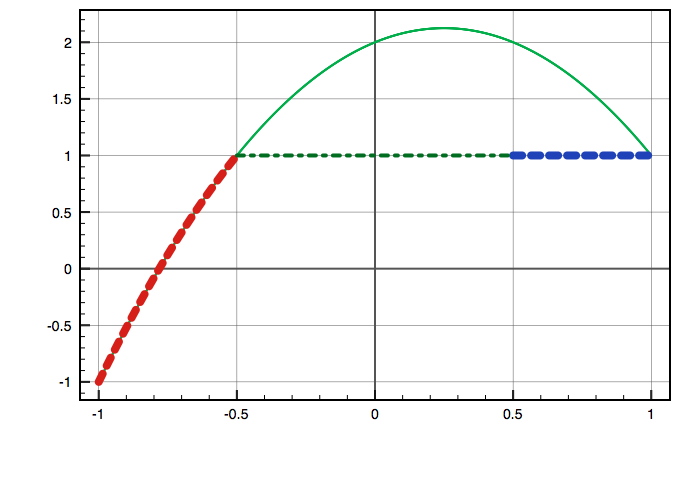}\quad
\includegraphics[width=7cm, height=5cm]{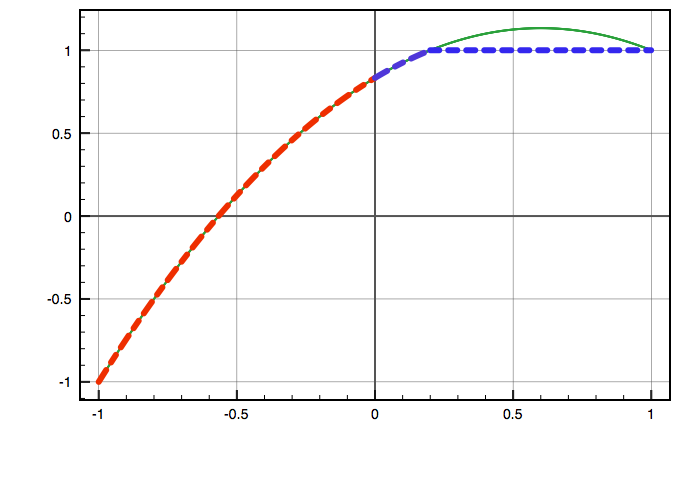}\\
\includegraphics[width=7cm, height=5cm]{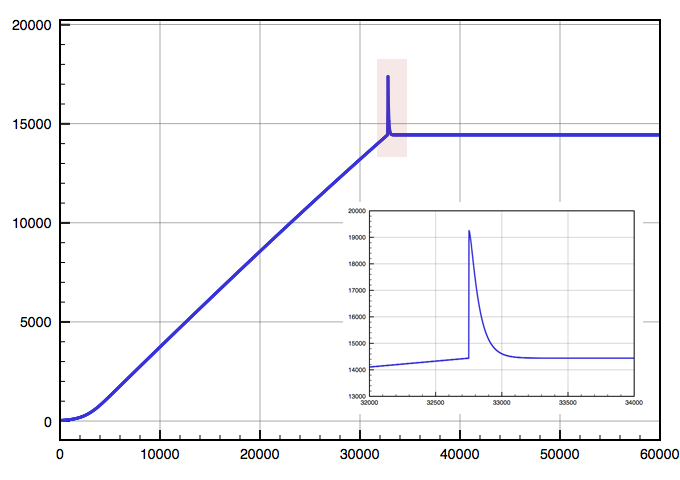}\quad
\includegraphics[width=7cm, height=5cm]{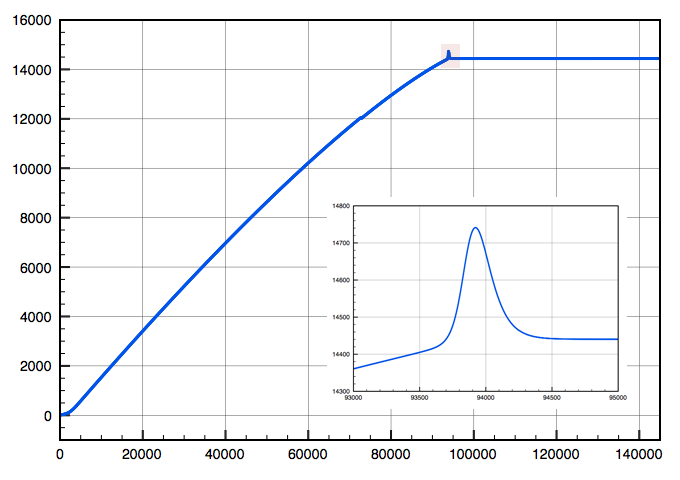}
\caption{\footnotesize Regime $\delta\to 2/\gamma$ and $[u]\to-\infty$.
Upper line: the phase plane $(x,y)$ with the graph of the function $G$ (continuous line) 
and the heteroclinic orbit connecting the points $P_-=(-1,-1)$ and $P_+=(+1,+1)$  
(dashed lines for the continuous part and dashed-dotted line for the jump).
Case $\gamma<2$ (left) and case $\gamma>2$ (right).
Lower line: corresponding temperature profiles for the case $\gamma<2$ (left)
and case $\gamma>2$ (right), both with an inside box showing a zoomed version
of the spike.}
\label{fig:casilimite}
\end{figure}

For $\delta\to 2/\gamma$, and $\gamma\geq 2$, then the unstable manifold $\mathcal{U}_-$ of
$P_-=(-1,-1)$ coincides with the graph of the function $G_{\gamma,\alpha}$ on the interval $[-1,0]$.
The stable manifold $\mathcal{S}_+$ of $P_-=(+1,+1)$ is given by the union of the horizontal segment
connecting the point $(\gamma-2,1)$ and $(1,1)$ and the graph of the function $G_{\gamma,\alpha}$
for $x\in[0,\gamma-2]$ (see Figure \ref{fig:casilimite}, right).

For $[u]\to0$, again in an appropriate scale, system \eqref{std} becomes
\[
		\frac{dx}{d\zeta} = 0,\qquad
		\frac{dy}{d\zeta}= \frac{1}{2}(1-x^2)
\]
In this regime, the limiting dynamics is very simple: the unstable manifold $\mathcal{U}_-$
of $P_-=(-1,-1)$ coincides with the vertical half-line $\{x=-1, y\geq -1\}$ and  the stable
manifold $\mathcal{S}_+$ of $P_-=(+1,+1)$ with the vertical half-line $\{x=1, y\leq 1\}$.
The heteroclinic orbit is described by a direct jump from $\mathcal{U}_-$ to $\mathcal{S}_+$.
The corresponding temperature profile is monotone for $\delta\to 0^+$ and non-monotone
for $\delta\to 2/\gamma$.

The intermediate cases for $\nu$ can be deduced qualitatively by noting that, for $\delta$
fixed and $\nu$ varying, the system \eqref{std} is a rotated vector field, as already observed
in Remark \ref{rem:rotated}.
In particular, as $\nu$ increases, that is as $|u_+-u_-|$ decreases, the trajectory determining
the manifolds $\mathcal{U}_-$ and $\mathcal{S}_+$ rotate counter clockwise passing from
one limiting configuration to the other.

\subsection*{Numerical experiments}
One of the advantage of the reduced system \eqref{std} resides in its semplicity and in 
the possibility of a numerical approximation of the structure of the radiative profiles 
by means of a standard solver for ordinary differential equations, complemented with
the conditions relative to the eventual jump point.
We present here a series of experiments where parameters have been chosen 
only in part realistic.
The value for the adiabatic constant $\gamma$ is taken equal to $5/3$ 
and the constant $R$ is set equal to $8.31$.
All of the constants relative to the radiative coupling, $\sigma, \sigma_s$ and
$\tau$, are chosen equal to 1.
The function $g$ has the form expressed in \eqref{specialpeg}.
As quoted, the realistic exponent $\alpha$ is $4$;
nevertheless, for computational advantage, we choose to set $\alpha=2$
(the case $\alpha=1$ has different behavior in the limiting regimes).
Also, to reduce the number of free parameters, we fix $\rho_-$ at
the value $0.1$. 

Finally, the different cases depend on the choices of the ratio $\delta$ and values
of $U_c$ (the average of the velocities $U_\pm$ at the left and at the right of the profile).
Since $\gamma=5/3$, the thresholds for $\delta$ are 
\[
	\delta_{\textrm{spike}}=0.4\qquad
	\delta_{\textrm{jump}}=0.8\qquad
	\frac{2}{\gamma}=1.2
\]
Being interested in the behavior for weak/strong radiation and weak/strong shocks
we choose $\delta=0.6$ (Figure \ref{fig:temprofile}, left column)  
and $\delta=1.0$ (Figure \ref{fig:temprofile}, right column)
and $U_c$ equal to $10, 50$ and $100$.
\begin{figure}[ht]\label{fig:temprofile}
\includegraphics[width=7cm, height=4cm]{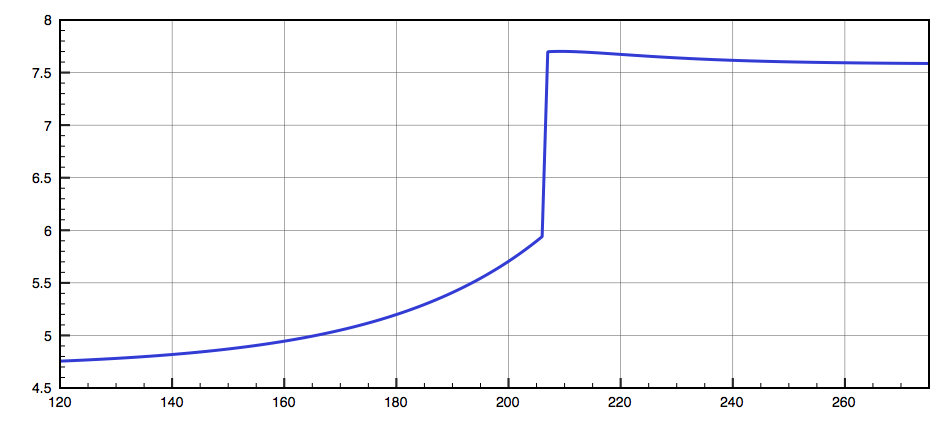}\quad
\includegraphics[width=7cm, height=4cm]{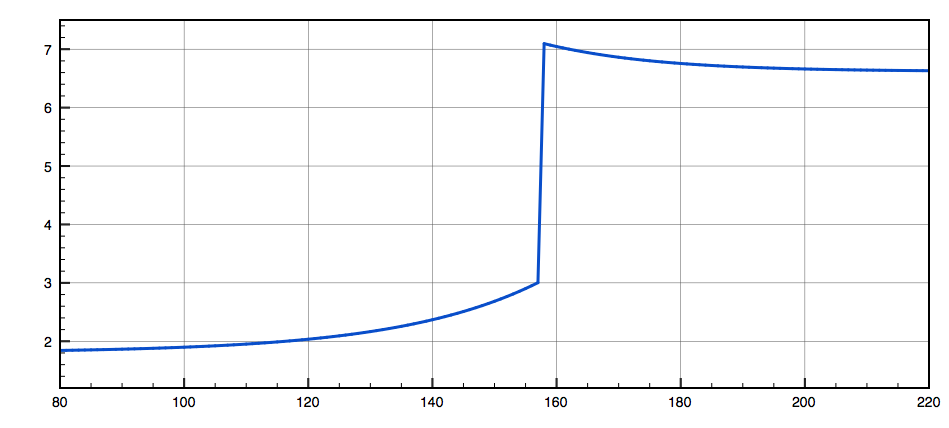}\\
\includegraphics[width=7cm, height=4cm]{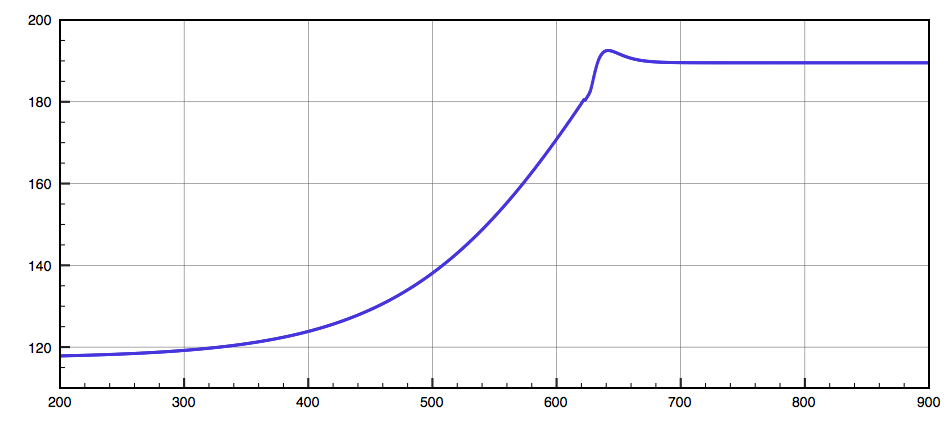}\quad
\includegraphics[width=7cm, height=4cm]{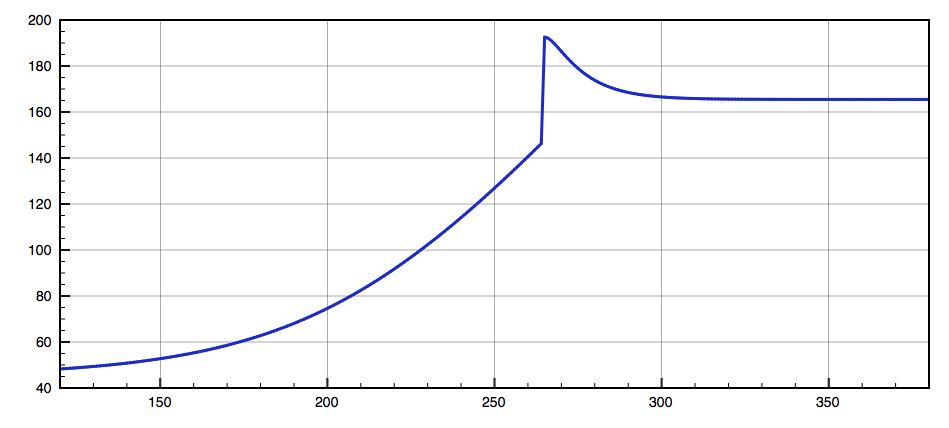}\\
\includegraphics[width=7cm, height=4cm]{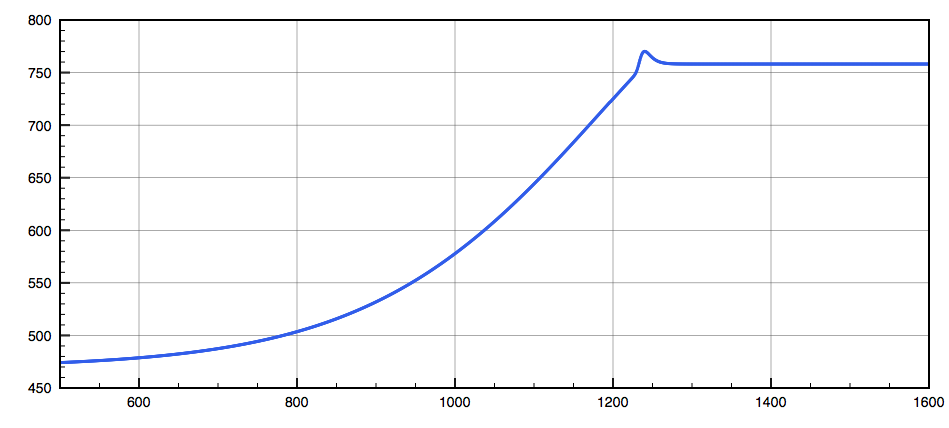}\quad
\includegraphics[width=7cm, height=4cm]{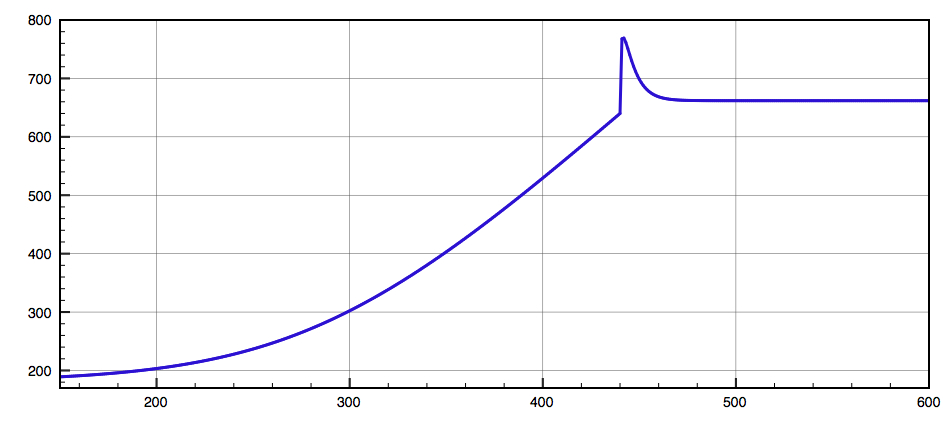}\\
\caption{\footnotesize Temperature profiles in the cases
$\delta=0.6$ (left column) and $\delta=1.0$ (right column), 
corresponding to the values  $U_c=10, 50, 100$ (increasing from top to bottom).}
\end{figure}

Coeherently with the result previously deduced, all of the profiles are non-monotone
since both the values of $\delta$ are taken above the threshold $\delta_{\textrm{spike}}=0.4$.
The choice $\delta=1.0$ is also above the threshold $\delta_{\textrm{jump}}=0.8$ and
thus the profile are discontinuous.
For all of them, the spike coincide with the jump point.
The choice $\delta=0.6$ is consistent with both continuous and discontinuous profiles
and this is revelad by the fact that for $U_c=10$, the temperature has a jump,
and for both $U_c=50$ and $U_c=100$ the profile is continuous.

Higher values for $\delta$, closer to the limiting value $2/\gamma$, give raise to pictures
resembling the one already presented in Figure \ref{fig:casilimite} (left column).
In particular, the width of the spike region shrinks to a single point.

\subsection*{Hamer model}
In the case of the system \eqref{hamer} with $f(s)=\frac12\,s^2$,
if $g$ is monotone increasing, $G(0)\leq -1$ if and only if $c\geq u_-$
and $G(0)\geq 1$ if and only if $u_+\leq c$.
Since $u_+<c<u_-$ for any admissible shock wave, discontinuous orbits
may appear only as a consequence of the  condition $G'(0)<2\nu^2$, that,
in the present case, translates into
\[
	\frac{dg}{du}\left(c\right)<\frac12\,\frac{[u]^3}{[g]}.
\]
Considering $g(u)=\sigma\,u$, $\sigma>0$, (see Example \eqref{ex:gburgers}),
the above condition becomes
\[
	|u_+-u_-|>\sqrt{2}\,\sigma,
\]
that describes the (sharp) threshold for discontinuous profiles (see \cite{KawaNish99}).

The case $g(u)=\sigma\,u^2$ gives the condition
\[
	|u_+-u_-|>2\sqrt{2}\,\sigma\,|c|
\]
showing that discontinuous profiles are possible also in the regime of small $|u_+-u_-|$,
if the sonic value $c$ is sufficiently small.
Note that in this case, the function $G$ is non-monotone in $[-1,1]$ if and only if 
\[
	|u_+-u_-|>2\,|c|,
\]
hence, if $\sigma>1/\sqrt{2}$, both non-monotonicity of $G$ and continuity of the profile
are compatible for some appropriate choice of the asymptotic states $u_\pm$.

\section{Conclusions}

The analysis performed shows that shock waves are robust patterns for the 
system \eqref{radeuler} under the assumption \eqref{specialpeg}.
Indeed, the existence of such structures is not limited to specific ranges for the asymptotic
states; oppositely, they do exist for any regime consistent with the reduced hyperbolic system,
obtained by disregarding the radiation effects.
Additionally, the system of algebraic-differential system can be dealth with in a rigorous manner
with a restricted number of technicalities and, at the same time, with a good number of quantitative
informations relative to the internal structure of the shock transition, specifically with respect to the
presence of change of monotonicity.
This allows to extend the analysis carried out in the region $\delta\to 2/\gamma$ in \cite{MihaMiha84} 
to any regime and gives a sound basis to the result formally derived in \cite{HeasBald63}.
Also, the observation that, for fixed $\delta$, the reduced system \eqref{std} depends on the
value $|U_+-U_-|$ (or, equivalently, on $U_c$) as a {\sf rotated vector field} permits, in principle, 
to follow the changes in the qualitative properties of the profiles in a ``monotone'' way with respect
to the variation of such value.

The relative simplicity and solidity of the structure of the problem is also revealed from the
fact that there is a great freedom in the possible form for the coupling term $g$ --though the power-like
form is the most accredited version-- and that there is a reduced number of parameters required 
to classify different forms for the profiles.
More general forms for pressure and internal energy should also preserve the existence of
the radiative profiles so that it is reasonable that the results presented in this article extend 
to more general fluids.

A natural subsequent step is the stability analysis for such radiation profiles, in particular in
the presence of spikes.
Such problem could be approached taking inspiration from other classes of systems of partial
differential equations in the context of conservation laws that support non-monotone traveling
wave as in the case of combustion waves.
Also, it would be very interesting to analyze the case of radiation hydrodynamics in more
general forms and specifically without taking advantage of the assumptions that reduce the
model to a hyperbolic-elliptic system of partial differential equations.
In particular, considering the native form of a continuous-kinetic coupled model would be
a very challenging and stimulating research project.

\end{document}